\documentclass{amsart}

\usepackage{amssymb, amsmath}
\usepackage[fleqn,tbtags]{mathtools}
\mathtoolsset{showonlyrefs,showmanualtags}

\usepackage{mathrsfs}
\usepackage{amscd}
\usepackage[active]{srcltx}
\usepackage{verbatim}

\usepackage{newsymbol}
 \let\emptyset \undefined
 \let\ge       \undefined
 \let\le       \undefined
 \newsymbol\le          1336  
 \newsymbol\ge          133E  
 \newsymbol\emptyset    203F
 \newsymbol\notle       230A
 \newsymbol\notge       230B

\theoremstyle{plain}
\newtheorem{theorem}{Theorem}[section]
\theoremstyle{remark}
\newtheorem{remark}[theorem]{Remark}
\newtheorem{example}[theorem]{Example}

\theoremstyle{plain}

\newtheorem{lemma}[theorem]{Lemma}
\newtheorem{proposition}[theorem]{Proposition}
\newtheorem{definition}[theorem]{Definition}

\newtheorem{assumption}[theorem]{Assumption}
\numberwithin{equation}{section}

\def\R{{\mathbb R}}

\newcommand{\E}{{\mathbb E}}
\renewcommand{\P}{{\mathbb P}}
\newcommand{\F}{{\mathscr F}}
\newcommand{\G}{{\mathscr G}}
\renewcommand{\H}{{\mathscr H}}

\renewcommand{\a}{\alpha}
\renewcommand{\b}{\beta}
\newcommand{\g}{\gamma}
\renewcommand{\d}{\delta}
\newcommand{\e}{\varepsilon}

\newcommand{\Om}{\Omega}

\newcommand{\embed}{\hookrightarrow}
\newcommand{\s}{^*}
\newcommand{\lb}{\langle}
\newcommand{\rb}{\rangle}
\newcommand{\calL}{{\mathscr L}}
\newcommand{\n}{\Vert}
\newcommand{\one}{{{\bf 1}}}

\allowdisplaybreaks

\begin{document}

\title[Equivalence of laws for SPDEs driven by an fBm]
{Equivalence of laws and null controllability for {\larger SPDE}s driven by
a fractional Brownian motion}

\author{Bohdan Maslowski}
\address{Faculty of Mathematics and Physics\\
Charles University in Prague\\
Sokolovska 83, 186 75 Prague 8\\
Czech Republic}	
\email{maslow@karlin.mff.cuni.cz}

\author{Jan van Neerven}
\address{Delft Institute of Applied Mathematics
\\Delft University of Technology
\\P.O. Box 5031, 2600 GA Delft
\\The Netherlands}
\email{J.M.A.M.vanNeerven@TUDelft.nl}



\date\today

\begin{abstract}
We obtain necessary and sufficient conditions for equivalence of law
for linear stochastic evolution equations driven by a general Gaussian noise by
identifying the suitable  space of controls for the corresponding deterministic
control problem.
This result is applied to semilinear (reaction-diffusion)
equations driven by a fractional Brownian motion. We establish the equivalence
of continuous dependence of laws of solutions to semilinear equations
on the initial datum in the topology of pointwise convergence of measures
and null controllability for the corresponding deterministic control problem.
\end{abstract}

\thanks{The second named author gratefully acknowledges the support by
VICI subsidy 639.033.604 in the `Vernieuwingsimpuls' programme of the
Netherlands Organization for Scientific Research (NWO). The first author was
supported by GACR grant no. P201/10/0752.}

\maketitle

\section{Introduction}\label{sec:1}

The equivalence (mutual absolute continuity) of probability distributions of
solutions to infinite-dimensional stochastic equations has been extensively
studied for equations driven by Wiener process and is of importance in the
investigation of large-time behaviour and ergodicity of solutions (see e.g.
\cite{MM,Ma1,PZ} for early works on stochastic reaction-diffusion
equations, the book \cite{DaZ2}, and the references therein). Indeed, by
Doob's theorem and its improvements \cite{Se,St} it enables proofs of
strong ergodicity and mixing of the Markov semigroups.

The first results in this direction for stochastic linear equations appeared in
the pioneering monograph by Da Prato and Zabczyk \cite{DaZ} (cf. also the earlier
paper by J. Zabczyk \cite{Za}), where the following statement may be found.
Consider the equation
\begin{equation}\label{I1}
\left\{
\begin{aligned}
 dX(t)& = AX(t)\,dt + B\, dW,\\ X(0)& =x,
\end{aligned}\right.
\end{equation}
in a real separable Hilbert space $E$, where $A$ is a densely defined linear
operator on $E$ generating a strongly continuous semigroup, $B$ is bounded on
$E$ and $W$ is a standard cylindrical Wiener process in $E$. Assuming the existence
of a unique $E$-valued mild solution, for a given $T>0$ the
probability laws of the solutions are equivalent for different values of the
initial datum $x\in E$ if and only if the corresponding deterministic controlled
system
\begin{equation}\label{I2}
\left\{
\begin{aligned}y' & = Ay + Bu, \\ y(0)&=x,
\end{aligned}\right.
\end{equation}
is exactly null controllable at $T$ with controls from the set $L^2(0,T;E)$,
i.e., for each $x\in E$ there exists a control $u\in L^2(0,T;E)$ such that
$y(T)=0$. This made it possible to utilize the well-developed deterministic
controllability theory for proving equivalence of laws for the stochastic equation
\eqref{I1}.

In the non-Markovian case the problem of equivalence of laws was addressed in
the paper \cite{DMP1} in the so-called diagonal case (when the operators $A$ and
$B$ commute) with a fractional Brownian motion as the driving process. Here we
extend the result of this paper by providing necessary and sufficient conditions
for the equivalence of laws (Example 3.4). This toy model shows the
interesting (but natural) feature that smoother paths of the driving process
(i.e. a bigger Hurst parameter) require more restrictive conditions for
equivalence of laws (which, of course, is just the converse with
respect to regularity problems).

In Section \ref{sec:2}
the first part of the paper we obtain an extension of the above-mentioned
result from \cite{DaZ, Za} to the case of general Gaussian noise. The state
space $E$ is a general Banach space and the objective is to identify the suitable
space of controls such
that the above-described relation between equivalence of laws for \eqref{I1} and
null controllability for \eqref{I2} holds true.

In Sections \ref{sec:3} and \ref{sec:4} these results are applied to semilinear (reaction-diffusion)
equations driven by a fractional Brownian motion; this
part is based on the density formula established in \cite{DMP3}. The
main result here is the equivalence of continuous dependence of laws of solutions to semilinear equations
on the initial datum in topology of pointwise convergence of measures
and null controllability for the corresponding control problem. As in
the Markov case, we call this property the {\em strong Feller property}.
The basic idea of the proof of this result is due to Maslowski and Seidler
\cite{MS}. In these results, for simplicity we assume that the space $E$ is Hilbert.
Using the ideas of \cite{BrzNeeSal}, the results can be extended to (type $2$) Banach
spaces, but our examples would not be improved much by this extension.

\section{Null controllability and equivalent of laws} \label{sec:2}

Let $\H$ be a real Hilbert space with inner product $[\cdot,\cdot]_\H$,
and let $W$ be an $\H$-isonormal process on some probability space
$(\Om,\F,\P)$.
By definition, this means that $W$ is a bounded linear mapping from $\H$ into
$L^2(\Om)$
such that the following two conditions are satisfied:
\begin{enumerate}
 \item for all $h\in \H$, the random variable $W(h)$ is centred Gaussian;
 \item for all $h_1,h_2\in\H$,
$$ \E W(h_1)W(h_2) = [h_1,h_2]_\H.$$
\end{enumerate}

Let $E$ be a real Banach space. We denote by $\g(\H,E)$ the completion of the
algebraic tensor product
$\H\otimes E$ with respect to the norm
$$ \Big\n \sum_{n=1}^N h_n \otimes x_n \Big\n_{\g(\H,E)}^2 = \E \Big\n
\sum_{n=1}^N \g_n x_n\Big\n^2,$$
where $(h_n)_{n=1}^N$ is an orthonormal sequence in $\H$, $(x_n)_{n=1}^N$ is a
sequence in $E$, and
$(\g_n)_{n=1}^N$ is Gaussian sequence, i.e., a sequence of independent standard
real-valued Gaussian random variables.
The natural inclusion mapping $\H\otimes E\embed \calL(\H,E)$ extends to a
continuous inclusion mapping
$\g(\H,E)\embed \calL(\H,E)$; this allows us to identify elements of $\g(\H,E)$
with bounded operators acting from
$\H$ to $E$. The operators belonging to $\g(\H,E)$ are called the $\g$-{\em
radonifying} operators from $\H$ to $E$.

The mapping $W:\H\to L^2(\Om)$ extends to an isometry, also denoted by $W$, from
$\g(\H,E)$
into $L^2(\Om;E)$ by putting
$$ W(h\otimes x) := W(h)\otimes x.$$
Moreover, $\lb W(R),x\s\rb = W(R\s x\s)$ for all $R\in \g(\H,E)$ and $x\s\in
E\s$.
This extension can be viewed as an abstract stochastic integral, as can be seen
by considering the case
$\H = L^2:= L^2(0,T)$ (see \cite{Nee, NeeWei}):

\begin{example}\label{ex:BM}
If $W$ is an $L^2$-isonormal process, then $B_t := W(\one_{(0,t)})$ defines a
standard real-valued
Brownian motion  $(B_t)_{t\in [0,T]}$ and
for all $h\in H$ we have
$$ W(h) = \int_0^T h(t)\,dB_t.$$
Moreover, the extension $W:\g(L^2,E)\to L^2(\Om;E)$ coincides with the
$E$-valued stochastic integral introduced in \cite{NeeWei}.
Conversely, if $(B_t)_{t\in [0,T]}$ is a standard real-valued
Brownian motion $(B_t)_{t\in [0,T]}$, then above identity defines an
$L^2$-isonormal process $W$.
\end{example}

Thus, $L^2$-isonormal process are in one-to-one correspondence with standard
real-valued Brownian motions.

\begin{example}
Let $H$ be a real Hilbert space and let $L^2(H):= L^2(0,T;H)$. An $H$-{\em
cylindrical Brownian motion}
is an $L^2(H)$-isonormal process $W$. For any real Banach space $E$, the
associated stochastic integral
$W: L^2(H)\to L^2(\Om)$ extends to an isometry $W$ from $\g(L^2(H),E)$ into
$L^2(\Om;E)$ (see \cite{NeeWei}).
\end{example}

In what follows it will be useful to have an explicit description of the
reproducing kernel Hilbert space (RKHS)
associated with the (centred Gaussian) random variables $W(R)$ associated with
$\g$-radonifying operators $R$:

\begin{proposition}\label{prop:rkhs}
Let $W$ be an $\H$-isonormal process and let $R\in \g(\H,E)$ be a
$\g$-radonifying operator.
The reproducing kernel Hilbert space $\G_R$ associated with $W(R)$ equals the
range of $R$.
\end{proposition}
\begin{proof} Let $Q_R\in\calL(E\s,E)$ denote the covariance operator of $W(R)$.
For all $x\s\in E\s$ we have
\begin{equation}\label{fake}
 \lb Q_R x\s,x\s\rb = \E |\lb W(R),x\s\rb|^2 = \E |W(R\s x\s)|^2 = \n R\s
x\s\n_{\H}^2 = \lb RR\s x\s,x\s\rb.
 \end{equation}
Hence $RR\s = Q_R = i_Ri_R\s,$ where $i_R:\G_R\embed E$ is the canonical
embedding.
It follows that $R$ maps ran$(R\s)$ into ran$(i_R) = \G_R$.
As an element of $\G_R$, $RR\s x\s$ equals $i_R\s x\s$, so
$$ \n RR\s x\s \n_{\G_R}^2 = \n i_R\s x\s\n_{\G_R}^2 = \lb Q_R x\s,x\s\rb = \n
R\s x\s\n^2.$$
It follows that $R$ extends to an isometry from $\overline{{\rm
Range}(R\s)}^{\H}$
onto  $\overline{{\rm Range}(i_R\s)}^{\G_R} = \G_R$.
Finally, if $h\perp\overline{{\rm Range}(R\s)}^{\H}$ then $Rh=0$.
\end{proof}

From this point on, we fix a Hilbert space $H$ and a finite time horizon
$0<T<\infty$, and write
$L^2(H):=L^2(0,T;H)$.
We fix a function $\Phi:(0,T)\to \calL(H,E)$ which has the property
that the adjoint orbits $t\mapsto \Phi\s(t)x\s$ belong to $L^2(H)$.
Here, $\Phi\s(t) = (\Phi(t))^* : E\s\to H$ denotes the adjoint of $\Phi(t): H\to
E$ defined
via the Riesz representation theorem.

In order to discuss stochastic integrability of
$\Phi$ with respect to the $\H$-isonormal process $W$, we need to connect the
spaces $\H$ and $L^2(H)$.
This will be done in the next two subsections, where we consider the situations
where we have
continuous dense embeddings $\H\embed L^2(H)$ and $L^2(H)\embed \H$,
respectively. These embeddings can
be interpreted as saying that $W$ is `rougher', respectively `smoother', than
$H$-cylindrical motions.
This rough noise case is slightly subtler to deal with and will therefore be
presented in detail;
the smooth noise case proceeds entirely analogous, with some slight
simplifications.

The basic examples we have in mind are provided by $H$-cylindrical (classical,
Liouville) fractional Brownian motions; see Section 3.

\subsection{The rough noise case}

In this subsection we make the following assumption.

\begin{assumption}\label{ass1}
The space $\H$ is continuously and densely included in $L^2(H)$.
\end{assumption}

We then have continuous and dense embeddings
$\H\embed L^2(H)\embed \H^\star,$
where $\H^\star$ denotes the dual of $\H$ relative to the $L^2(H)$-duality.
Thus, for all $h\in \H$ and $f\in L^2(H)$ we have
$$\lb h, f\rb_{\lb \H,\H^\star\rb} = \int_0^T h(t)f(t)\,dt.$$
For each $h\in\H$ we define the element
 $\phi_h\in \H^\star$ by
$$ \lb g,\phi_h\rb_{\lb \H,\H^\star\rb} = [g,h]_{\H}.$$
By the Riesz representation theorem, the correspondence $h\leftrightarrow
\phi_h$ sets up
a bijective correspondence between $\H$ and $\H^\star$.

For a mapping $S$ from a Banach space $X$ into $ \H$ we denote by $S^\star:
\H^\star \to X\s$ the adjoint, so that
for all $x\in X$ and $h^\star\in \H^\star$ we have
$$\lb x, S^\star h^\star\rb = \lb Sx, h\rb_{\lb \H, \H^\star\rb}.$$

\begin{definition} The function $\Phi:(0,T)\to \calL(H,E)$ is said to be
{\em stochastically integrable} with respect to the $\H$-isonormal
process $W$ if $t\mapsto \Phi\s(t) x\s$
belongs to $\H$ for all $x\s\in E\s$ and
there exists an operator $R_\Phi\in \g(\H,E)$ such that
$R_\Phi\s x\s  = \Phi\s x\s$ for all $x\s\in E\s$.
The random variable $W(R_\Phi)$ is called the {\em stochastic integral} of
$\Phi$ with respect to $W$, notation
$$\int_0^T \Phi\,dW = W(R_\Phi).$$
\end{definition}

Here, $R_\Phi^*: E\s\to \H$ denotes the adjoint of $R_\Phi:\H\to E$ defined
via the Riesz representation theorem.

\begin{proposition}\label{prop:R}
Let Assumption \ref{ass1} hold and suppose $\Phi$ is stochastically integrable
with respect to
$W$. Define the bounded operator $R: \H^\star\to E^{**}$ by
$$ \lb x\s, R h^\star \rb := \lb \Phi\s x\s, h^\star\rb_{\lb \H,\H^\star\rb}.$$
Then  $ R = R_\Phi^{*\star}$, and both operators map $\H^\star $ into $E$.
\end{proposition}
\begin{proof}
For all $f\in L^2(H)$ we have
\begin{equation}\label{eq:identities}
\begin{aligned}
\ \lb x\s, R_\Phi^{*\star}f\rb
 = \lb R_\Phi\s x\s, f\rb_{\lb\H,\H^\star\rb}
 & = \lb\Phi\s x\s, f\rb_{\lb\H,\H^\star\rb}\\ & = \int_0^T [\Phi\s(t) x\s,
f(t)]_H\,dt = \lb R f,x\s\rb.
\end{aligned}
\end{equation}
This proves that $R = R_\Phi^{*\star}$ as operators from $H^\star$ to $E^{**}$.
To prove that $R$ (and hence also $R_\Phi^{*\star}$) maps $\H $ into $E$ it
suffices to prove that for all
$h\in \H$ we have
\begin{equation}\label{eq:id}
 R_\Phi h =   R_\Phi^{*\star}\phi_h
\end{equation}
in $E^{**}$,
and then to observe that $R_\Phi$ takes values in $E$.
To prove this identity, note that for all $x\s\in E\s$ we have
$$ \lb R_\Phi h, x\s\rb_{\lb E,E\s\rb} = [R_\Phi\s x\s,h]_\H = \lb R_\Phi\s x\s,
\phi_h\rb_{\lb\H,\H^\star\rb} =
\lb x\s, R_\Phi^{*\star}\phi_h\rb_{\lb E\s, E^{**}\rb}.$$
\end{proof}

It follows from the proposition that, under the stated assumptions, $t\mapsto
\Phi(t)f(t)$ is
Pettis integrable and, for all $f\in L^2(H)$,
\begin{equation}\label{eq:Rf}
 Rf = \int_0^T \Phi(t)f(t)\,dt.
\end{equation}

\medskip
Now let $A$ generate a $C_0$-semigroup $(S(t))_{t\ge 0}$ on $E$ and let $B\in
\calL(H,E)$ be a fixed operator.
We are interested in the control problem
\begin{equation}\label{eq:contr}
\begin{aligned}
 x' & = Ax+Bf,\\
 x(0) & = x_0,
\end{aligned}
\end{equation}
where $f$ is a `rough' control, that is, we assume that $f\in\H^\star$.
For controls $f\in L^2(H)$ the mild solution $x$ of the control problem at time
$T$ is given by
$$ x(T) = S(T)x_0 + \int_0^T S(T-t)Bf(t)\,dt.$$
If we assume that $\Phi = S(T-\cdot)B$ satisfies the assumptions of Proposition
\ref{prop:R}, the identity
\eqref{eq:Rf} takes the form
$$ Rf = \int_0^T S(T-t)Bf(t)\,dt.$$
Accordingly, for controls $f\in H^\star$ we define the mild solution of the
problem \eqref{eq:contr}
at time $T$ to be
$$ x(T) = S(T)x_0 + Rf,$$
where $R: \H^\star\to E$ is the map of Proposition \ref{prop:R}.

\begin{theorem}\label{thm:NC1}
Let Assumption \ref{ass1} hold and let $W$ be an $\H$-isonormal process.
Suppose $S(T-\cdot)B$ is stochastically integrable with respect to $W$.
Let $\mathscr{G}_T$ be the RKHS associated with the stochastic integral
$\int_0^T S(T-t)B\,dW(t)$.
The following assertions are equivalent:
\begin{enumerate}
\item[\rm(i)] $S(T)E\subseteq \mathscr{G}_T$;
\item[\rm(ii)] For all $x_0\in E$ the problem
\begin{align*}
\begin{cases}
 x' & = Ax+Bf,\\
 x(0) & = x_0,
\end{cases}
\end{align*}
is null controllable in time $T$ with a control $f\in \H^\star$.
\item[\rm(iii)] For all $x_0\in E$ the laws of the processes solving the
stochastic evolution equation
\begin{align*}
\begin{cases}
 dX(t) & = AX(t) +B\,dW,\\
  X(0) & = x_0,
\end{cases}
\end{align*}
are mutually absolutely continuous.
\end{enumerate}
\end{theorem}

\begin{proof}
(i) $\Leftrightarrow$ (ii):
The problem in (ii) is null controllable in time $T$
with control $f\in \H^\star$ if and only if $Rf =-S(T)x_0.$
Now, $f = \phi_h$ for some $h\in \H$ and
$$R f= R\phi_h = R_\Phi^{*\star}\phi_h = R_\Phi h.$$
Furthermore, the range of $R_\Phi$ equals the RKHS $\mathscr{G}_T$ by
Proposition \ref{prop:rkhs}.
Combining things, we see that the problem in (ii) is null
controllable in time $T$ if and only $S(T)x_0 \in \mathscr{G}_T$.

(i) $\Leftrightarrow$ (iii): This is a special case of the Feldman-Hajek theorem
on equivalent of Gaussian measures.
\end{proof}

\subsection{The smooth noise case}

Next we consider the following dual assumption.

\begin{assumption}\label{ass2}
The space $L^2(H)$ is
is continuously and densely included in $\H$.
\end{assumption}

We then have continuous and dense embeddings
$\H^\star\embed L^2(H)\embed \H$.

\begin{definition} The function $\Phi:(0,T)\to \calL(H,E)$ is said to be {\em
stochastically integrable} with respect to the $\H$-isonormal
process $W$ if
there exists an operator $R_\Phi\in \g(\H,E)$ such that
$R_\Phi\s x\s  = \Phi\s x\s$ for all $x\s\in E\s$.
The random variable $W(R_\Phi)$ is called the {\em stochastic integral} of
$\Phi$ with respect to $W$, notation
$$\int_0^T \Phi\,dW = W(R_\Phi).$$
\end{definition}

Note that, since we are assuming that the dual orbits $t\mapsto \Phi\s(t) x\s$
belong to $L^2(H)$, they automatically define elements of $\H$ (unlike in the
rough noise case).

\begin{proposition}\label{prop:R-dual}
Let Assumption \ref{ass2} hold and suppose $\Phi$ is stochastically integrable
with respect to
$W$. Define the bounded operator $R: \H^\star\to E^{**}$ by
$$ \lb x\s, R h^\star \rb := \lb \Phi\s x\s, h^\star\rb_{\lb \H,\H^\star\rb}.$$
Then  $ R = R_\Phi^{*\star}$, and both operators map $\H^\star $ into $E$.
\end{proposition}
\begin{proof}
The proof follows the lines of that of Proposition \ref{prop:R} {\em verbatim},
except that the identities
\eqref{eq:identities} hold only for elements $f\in \H^\star$.
\end{proof}

\medskip
Now let $A$ generate a $C_0$-semigroup $(S(t))_{t\ge 0}$ on $E$ and let $B\in
\calL(H,E)$ be a fixed operator.
In the present setting, no ambiguities with regard to the definition of a mild
solution for the
control problem arise and we have the following result.

\begin{theorem}\label{thm:NC2}
Let Assumption \ref{ass1} hold and let $W$ be an $\H$-isonormal process.
Suppose $S(T-\cdot)B$ is stochastically integrable with respect to $W$.
Let $\mathscr{G}_T$ be the RKHS associated with the stochastic integral
$\int_0^T S(T-t)B\,dW(t)$.
The following assertions are equivalent:
\begin{enumerate}
\item[\rm(i)] $S(T)E\subseteq \mathscr{G}_T$;
\item[\rm(ii)] For all $x_0\in E$ the problem
$$
\begin{aligned}
 x' & = Ax+Bf,\\
  x(0) & = x_0,
\end{aligned}
$$
is null controllable in time $T$ with a control $f\in \H^\star$.
\item[\rm(iii)] For all $x_0\in E$ the laws of the processes solving the
stochastic evolution equation
$$
\begin{aligned}
 dX(t) & = AX(t) +B\,dW,\\
  X(0) & = x_0,
\end{aligned}
$$
are mutually equivalent.
\end{enumerate}
\end{theorem}

\begin{proof}
The proof follows the lines of that of Theorem \ref{thm:NC1} {\em verbatim}.
\end{proof}

\section{Fractional Ornstein-Uhlenbeck processes}\label{sec:3}

\subsection{Fractional Brownian motion}
In the present section the results of the preceding part are applied to the
fractional Ornstein-Uhlenbeck process, i.e. to the linear stochastic evolution
equation in which the driving process is a classical space-dependent fractional
Brownian motion (fBm)
in time and white or, possibly, correlated in space. At first we recall standard
definitions of these objects and explain how they may be understood in the
framework developed above.

We begin with the case of a scalar-valued fBm $B^\beta = (B^\beta(t))_{t\in
[0,T]}$ with Hurst
parameter $\beta\in (0,1)$.
Following the approach taken in \cite{BrzNeeSal} we
identify $B^\b$ with an $\H_\beta$-isonormal process $W^\b$, the inner product
of the
real Hilbert space $\H_\beta$ being given, for
step functions $\varphi_1, \varphi_2:[0,T]\to \R$, by
\begin{equation}
[\varphi_1,\varphi_2]_{\H_\b} = \mathbb{E}W^\beta(\varphi_1)W^\beta(\varphi_2)=
\mathbb{E}\int_0^T \!\varphi_1 \,dB^\beta\int_0^T\!\varphi_2\,dB^\beta
=[\mathscr{K}_\b^*\varphi_1,\mathscr{K}_\b^*\varphi_2]_{L^2(H)}.
\end{equation}
Here, the operator $\mathscr{K}_\b^*: \H_\b\to L^2(H)$ is defined, for step
functions
$\varphi:[0,T]\to\R$, by
\begin{equation}\label{eq:2.6}
\mathscr{K}_\b^* \varphi(t) = \varphi (t) K_\b(T,t) + \int_t^T \left(
\varphi(s) - \varphi(t) \right) \frac{\partial K_\b}{\partial s}
(s,t) \, ds,
\end{equation}
where $K_\b$ is the real-valued kernel
\begin{equation}\label{eq:2.1}
K_\b(t,s) = \frac{\tilde{c}_\beta(t-s)^{\beta-\frac{1}{2}}}{\Gamma\left(\beta +
\frac{1}{2}\right)} + \frac{\tilde{c}_\beta\left( \frac{1}{2} - \beta
\right)}{\Gamma\left(\beta + \frac{1}{2}\right)} \int_s^t \left( u-s
\right)^{\beta-\frac{3}{2}} \Big( 1-
\Big(\frac{s}{u}\Big)^{\frac{1}{2}-\beta}\Big) \, du,
\end{equation}
$\tilde{c}_\beta$ being a constant depending only on $\beta$. We
conclude that $\H_\b$ is the completion of the linear space of
step functions with respect to the norm
\begin{equation}\label{norm}\n \varphi \n _{\H_\b}
:= \n\mathscr{K}_\b^* \varphi \n_{L^2(H)}.\end{equation}
To give a more specific description
of this space it is convenient to distinguish two cases, corresponding to the
``rough'' and
``smooth'' noise cases considered above.

For $0<\beta<\frac12$ (the rough noise case) we have
\begin{equation}\label{eq:X2}
(\mathscr K_\b^* \varphi) (t) = c_\beta t^{\frac12-\beta} D^{\frac12
-\beta}_{T_-}
(u_{\beta-\frac12}\varphi) (t),\quad \varphi \in \H_\b ,
\end{equation}
where $u_{\alpha}(s) = s^{\alpha}$ and $D^{\alpha}_{T_-}$ is the right-sided
fractional derivative
\begin{equation}
\left(D_{T-}^\alpha \varphi\right) (t) = \frac{1}{\Gamma(1-\alpha)}
\Big( \frac{\varphi(t)}{(T-t)^\alpha} + \alpha \int_t^T
\frac{\varphi(s) -\varphi(t)}{(s-t)^{\alpha+1}}\, ds \Big)
\end{equation}
and $c_\beta$ is a constant depending only on $\b$. It is not difficult to see
that the
space $\H_\b$,  as a set, may be identified as
 \begin{equation}
 \H_\b  = I_{T_-}^{\frac12-\beta}(L^2(0,T))
 \end{equation}
(cf. \cite[Proposition 6]{AMN}).

For $\frac12<\beta<1$ (the smooth noise case) we have
\begin{equation}\label{eq:X3}
(\mathscr K_\b^*\varphi) (t) = c_\beta t^{\frac12-\beta} I^{\beta-\frac12}_{T_-}
(u_{\beta-\frac12}\varphi)(t),\quad \varphi \in \H_\b ,
\end{equation}
where $I^{\alpha}_{T_-}$ is the right-sided Riemann-Liouville integral,
\[
(I^{\alpha}_{T_-}\varphi) (t) = \frac1{\Gamma(\alpha)} \int^T_t (s-t)^{\alpha-1}
\varphi (s)\, ds.
\]
In this case we have a continuous dense embedding
$L^2(H)\embed \H_\b$ and the operator
${\mathscr K}_\b^*$
restricts to a bounded operator on $L^2(H)$ (cf. \cite{PT}). Moreover,
we have continuous dense embeddings
$$L^{\frac1{\beta}} (0,T,H)\embed\overline{\Theta}_\beta\embed \H_\b,$$ where
$\overline{\Theta}_\beta$ consists of those
elements $\psi$ from $L^1(0,T,H)$ such that
\[
\|\psi\|^2_{\overline{\Theta}_\beta} = \int^T_0\int^T_0|\psi(s)| \cdot |\psi(r)|
 \phi_\b (r-s) \,dr\,ds <\infty
\]
with $\phi_\b(r) = (2\b-1)\beta|r|^{2\beta-2}$.

\subsection{Cylindrical fractional Brownian motion}
Let $H$ be a real Hilbert space. An {\em $H$-cylindrical fBm} with Hurst
parameter $\b$ is defined as a
an $\H_\b\overline{\otimes}H$-isonormal process. Here,
$\H_\b\overline{\otimes}H$ denotes the
Hilbert space tensor product of $\H_\b$ and $H$. Under the identification made
at the beginning
of the previous subsection, an $\R$-cylindrical fBm is just a classical
scalar-valued
fBm (with the same Hurst parameter).

If $H$ has an orthonormal basis $(h_n)_{n\ge 1}$, one may think of an
$H$-cylindrical
fBm $B^\b$ as a formal series
\begin{equation}\label{eq:X1}
W^\beta_t = \sum^{\infty}_{n=1} (B^\beta_n)_t h_n, \quad t\in \mathbb R_+,
\end{equation}
where $(B_n^\beta)_{n\ge 1}$ is a sequence of independent fBm's with Hurst
parameter $\beta$.

\subsection{Ornstein-Uhlenbeck processes associated with a cylindrical fBm}
Let $H$ be a real Hilbert space, $E$ a real Banach space, and let $W^\b$ an
$H$-cylindrical fBm.
We consider the equation
\begin{equation}\label{eq:X77}
\begin{aligned}
dZ_t^x &= AZ_t^x\, dt + B \, dW^\beta_t, \\
Z_0^x &= x,
\end{aligned}
\end{equation}
where $A$ generates a strongly
continuous semigroup $S = (S(t))_{t\ge 0}$ on $E$,
the operator $B$ is bounded and linear from $H$ into $E$, and the initial datum
$x$ belongs to $E$.
The solution is
understood in the mild sense, i.e.
\begin{equation}\label{SDE}
Z_t^x = S(t) x + \int^t_0 S(t-r) B \, d W^\beta_r = : S(t) x + Z_t, \ t\in
[0,T],
\end{equation}
provided the stochastic integral is well-defined. For Liouville cylindrical fBm,
a necessary and sufficient condition for this is
given in \cite{BrzNeeSal}; for $0<\beta<\frac12$ the same condition works for
(classical)
cylindrical fBm.

Let $\g(H,E)$ denote the space of $\gamma$-radonifying operators from $H$ into
$E$.
For later reference we recall that if $E$ is a Hilbert space, then $\g(H,E) =
{\mathscr{L}}_2(H,E)$,
the space of Hilbert-Schmidt operators from $H$ to $E$, with equals norms.

A standard sufficient condition for existence and regularity of the so-called
Ornstein-Uhlenbeck process $(Z_t)_{t\in [0,T]}$ is recalled in the following
proposition.
We shall need it only in the case that $E$ is a Hilbert space and refer to
\cite{DMP1, DMP2} for a proof for the case $\theta=0$. For reasons of
completeness we shall
include a proof, which is an adaptation of the argument in
\cite[Theorem 5.5]{BrzNeeSal} (where more details are provided).

If the semigroup $S$ is analytic we may find  $z_0 >0$ sufficiently
large such that the fractional powers of $z_0-A$ exist ($z_0$ is fixed in the
sequel)
and and we denote by $E^{\theta}$ is the domain of the fractional
power $(z_0-A)^{\theta}$ equipped with the graph norm.

\begin{proposition}\label{XE}
Suppose $S$ is a strongly continuous analytic semigroup on the real Hilbert
space
$E$, let $B\in {\mathscr {L}}(H,E)$ be bounded,
and assume that
for some $\theta\in [0,1)$ and $\lambda \ge 0$ we have
\begin{equation}\label{eq:X8}
\n S(t)B\n_{{\mathscr{L}}_2(H,E^\theta)}  \ \le ct^{-\lambda}, \ t\in(0,T],
\end{equation}
for some $c\ge 0$.
If $$ \d+\theta +\lambda < \b,$$
the stochastic convolution process $(Z_t)_{t\in [0,T]}$ defined by \eqref{SDE}
is well-defined and
has a modification with paths in $C^{\delta}([0,T];E^\theta)$.
\end{proposition}

\begin{proof}
Fix  $0\le s\le t\le T$. By the triangle inequality in $L^2(\Om;E)$,
\begin{equation}\label{terms}
\begin{aligned}
\big(\E \n Z(t)-Z(s)\n_{E^\theta}^2\big)^\frac12
 & \le \Big(\E \Big\n\int_0^s [S(t-r)-S(s-r)]B\,dW^\b(r)
\Big\n_{E^\theta}^2\Big)^\frac12
\\ & \qquad\qquad\quad\  + \Big(\E \Big\n \int_s^t S(t-r)B\,dW^\b(r)
\Big\n_{E^\theta}^2\Big)^\frac12.
\end{aligned}
\end{equation}
We shall estimate both terms separately.

Fix $\tau\ge 0$ such that
$$ \lambda < \tau < \beta-\delta - \theta.$$
Then (with generic constants $c$)
$$ \n (z_0-A)^{-\tau}B\n_{{\mathscr{L}}_2(H,E^\theta)} \le c \int_0^\infty
t^{\tau-1}\n S(t)B\n_{{\mathscr{L}}_2(H,E^\theta)}\,dt
\le c \int_0^\infty t^{\tau-1-\lambda}\,dr <\infty. $$

For the first term in \eqref{terms} we have, for any
$\e>0$ such that $\delta+\tau+\theta+\e<\b$,
\begin{align*}
\ & \E \Big\n \int_0^s [S(t-r)-S(s-r)]B\,dW^\b(r) \Big\n_{E^\theta}^2
\\ & \simeq \E \Big\n \int_0^s
(s-r)^{\delta+\tau+\theta+\e}(z_0-A)^{\delta+\tau+\theta} S(s-r)
\\ & \hskip3cm \times
(s-r)^{-\delta-\tau-\theta-\e} [S(t-s)-I](z_0-A)^{-\delta-\tau}B\,dW^\beta(r)
\Big\n^2
\\ & {\le} c^2 \E \Big\n \int_0^s (s-r)^{-\delta-\tau-\theta-\e}
[S(t-s)-I](z_0-A)^{-\delta-\tau}
B\,dW^\b(r) \Big\n^2
\\ & {=} c^2 \n [S(t-s)-I](z_0-A)^{-\delta-\tau} B\n_{{\mathscr{L}}_2(H,E)}^2
\E\Big|\int_0^s
(s-r)^{-\delta-\tau-\theta-\e}\,dW^\b (r)\big|^2
\\ & {=}  c^2 s^{2\b-2\delta-2\tau-2\theta-2\e} \n
[S(t-s)-I](z_0-A)^{-\delta}(z_0-A)^{-\tau}B\n_{{\mathscr{L}}_2(H,E)}^2
\\ & {\le} c^2(t-s)^{2\delta}\n(z_0-A)^{-\tau} B\n_{{\mathscr{L}}_2(H,E)}^2,
\end{align*}
where the numerical value of $c$ changes from line to line (and is allowed to
depend on $T$).

Similarly,
$$
\begin{aligned}
\ &\E \Big\n \int_s^t S(t-r)B\,dW^\b(r) \Big\n_{E^\theta}^2
\\ & \quad \eqsim \E \Big\n \int_s^t (t-r)^{\theta+\tau+\e}
(z_0-A)^{\theta+\tau}S(t-r)
(t-r)^{-\theta-\tau-\e}(z_0-A)^{-\tau}B\,dW^\b(r)
\Big\n^2
\\ & \quad \le c^2 \E \Big\n \int_s^t (t-r)^{-\tau-\theta-\e}(z_0-A)^{-\tau}
B\,dW^\b(r)\Big\n^2
\\ & \quad = c^2 \n (z_0-A)^{-\tau}B\n_{{\mathscr{L}}_2(H,E)}^2 \E\Big|\int_s^t
(t-r)^{-\tau-\theta-\e}\,dW^\b(r)\Big|^2
\\ & \quad \le c^2 \n (z_0-A)^{-\tau} B\n_{{\mathscr{L}}_2(H,E)}^2
(t-s)^{2\beta-2\tau-2\theta-2\e}
\\ & \quad \le c^2 \n (z_0-A)^{-\tau} B\n_{{\mathscr{L}}_2(H,E)}^2
(t-s)^{2\delta}.
\end{aligned}
$$
Combining these estimates, this gives
$$\E \n Z(t)-Z(s)\n_{E_\theta}^2 \le c ^2 \n (z_0-A)^{-\tau}
B\n_{{\mathscr{L}}_2(H,E)}^2 (t-s)^{2\delta}
.$$

The assertion now follows from a routine application of the
Kahane-Khintchine inequalities (to pass from moments of order $2$ to moments of
order $p$) and the Kolmogorov-Chentsov continuity theorem.
\end{proof}

Now we are ready to formulate the main result for the fractional
Ornstein-Uhlenbeck process.
\begin{theorem}\label{thm:Y1}
Suppose $A$ generates a strongly continuous analytic semigroup $S$ on $E$.
Suppose furthermore that $B\in {\mathscr {L}}(H,E)$ is bounded and injective.
Assume:
\begin{itemize}
\item[\rm(i)] $\n S(t)B\n_{\mathscr \g(H,E)} \le ct^{-\lambda},\ {t\in (0,1)}$,
for some $\lambda \in [0,\beta)$,
\item[\rm(ii)]
$
\mbox{\rm Range} \, (B) \supset\mbox{\rm
Dom}((z_0-A)^{\mu})
$
for some $\mu \in [0,1-\beta )$.

\end{itemize}
Then there exists a continuous mild solution
$Z^x $ to the equation \eqref{eq:X77}, and for each $T>0$ the probability laws
of $Z^x_T$, $x\in E,$ are equivalent.
\end{theorem}

\begin{proof}
Since $B$ is injective and $S(t)$ maps $E$ into $D(A)$ for all $t>0$, the
operators
$B^{-1}S(t)$ are well-defined, and by the closed graph theorem they are bounded.
By (ii),
\begin{equation}\label{(ii)}
\n B^{-1} S(t)\n _{\mathscr L(E,H)} \le \n B^{-1} (z-A)^{-\mu}\n _{\mathscr
L(E,H)}
\n (z-A)^{\mu}S(t)\n _{\mathscr L(E)} \le \frac{c}{t^{\mu}},
\end{equation}
where the last step uses the analyticity of the semigroup $S$.

a) {\em The case $0<\beta <\frac12$.}
By Proposition \ref{XE} and Theorem \ref{thm:NC1} it remains to show the null
controllability of the equation
\[
\dot y = Ay + Bu
\]
on $[0,T]$ for the space of controls $\mathscr H^\star$. Note that
\[
\n\varphi\n_{\mathscr H} = \n\mathscr K_\b^*\varphi\n_{L^2(H)}.
\]
Hence for $\varphi \in \mathscr H^\star$ we have
\begin{align*}
\n \varphi\n_{\mathscr H^\star} &= \sup_{\n h\n_{\mathscr H} \le 1} |[\varphi,
h]_{L^2(H)}| = \sup_{\n g\n_{L^2(H)}\le 1} \Big| \int^T_0 [\varphi,
(\mathscr K_\b^*)^{-1} g]_H \,ds\Big| \\
&= \sup_{\n g\n_{L^2(H)}\le 1} \Big| \int^T_0 [\mathscr K_\b^{-1} \varphi,
g]_H \, ds \Big| = \Big( \int^T_0 \n\mathscr K_\b^{-1} \varphi\n^2_H
\,ds\Big)^{\frac12},
\end{align*}
taking into account  that $((\mathscr K_\b^*)^{-1})^* = \mathscr K_\b^{-1} =
c_{\beta} u_{\beta - \frac12} I^{\frac12-\beta}_{0+} u_{\frac12 - \beta}$, where
$c_{\beta}$
is a constant and $I_{0+}^\alpha$ is the left-sided fractional integral,
\begin{equation}
\left( I_{0+}^\alpha \varphi \right)\left(t\right) =
\frac{1}{\Gamma(\alpha)}\int_0^t \left( t-s \right)^{\alpha-1} \varphi(s) \,
ds,\quad \varphi \in L^2(H).
\end{equation}
Moreover, by \eqref{(ii)} and the fact that the assumption on $\mu$ implies
$\b-\frac12 +\mu<1$,
\begin{align*}
\n\mathscr K_\b^{-1} B^{-1}S(\cdot))(t)\n &\le ct^{\beta -\frac12} \int^t_0
\frac{\n B^{-1}S(s)\n s^{\frac12 -\beta}}{(t-s)^{\beta+\frac12}}\, ds\\
& \le ct^{\beta -\frac12} \int^t_0 \frac{s^{\frac12 -\beta -
\mu}}{(t-s)^{\beta+\frac12}}\,ds = \frac{c}{t^{\beta-\frac12 + \mu}},
\end{align*}
where $c\ge 0$ is a generic constant whose value is allowed to change from line
to line.
Since by assumption $\mu<1-\b$, this shows that
$\mathscr K_\b^{-1} (B^{-1} S(T-\cdot))\in L^2(H)$, and therefore the control
$\hat u_s(t) =
-\frac1{T} B^{-1} S(t)x$ steering  $x$ to zero belongs to the space $\mathscr
H^\star$.

b) {\em The case $\frac12 <\beta <1$.}
As in the previous case we only need to show the null controllability of
the equation
\[
\dot y = A + Bu
\]
on $[0,T]$ for the space of controls $\mathscr H^\star$. Using Theorem
\ref{thm:NC2} and proceeding as in a) it suffices to show that
\begin{equation}\label{INT}
\int^T_0 \n\mathscr K_\b^{-1} \varphi(s)\n^2_H \, ds < \infty
\end{equation}
where $\varphi(s) : = B^{-1} S(s)x, \quad x\in E$, is fixed and
\[
\mathscr K_\b^{-1} = c_{\beta} u_{\beta-\frac12} D^{\beta - \frac12}_{0+}
u_{\frac12-\beta},
\]
and $D^\alpha_{0+}$ denotes the left-sided fractional derivative.
\begin{equation}
\left(D_{0+}^\alpha \psi \right) (t) = \frac{1}{\Gamma(1-\alpha)} \Big( \frac
{\psi(t)}{t^\alpha} + \alpha \int_0^t \frac{\psi(t) - \psi(s) }
{(t-s)^{\alpha+1}} \, ds \Big),\quad \psi \in \H_\b .
\end{equation}
We have, for $\b-\frac12<\d<\frac12$ and $z_0>0$ large enough,
\begin{align*}
&\int^T_0 \n\mathscr K_\b^{-1} \varphi(t)\n^2_H\, dt
\\ & \le c \int^T_0
\frac{\n B^{-1}S(t)x\n^2_H t^{2\beta - 1} \cdot t^{1-2\beta}}{t^{2\beta-1}}\, dt
\\
& \qquad + c\int^T_0 t^{2\beta -1} \Big(\int^t_0
\frac{\n B^{-1}S(t)x-B^{-1}S(s)x\n _H s^{\frac12-\beta}}{(t-s)^{\beta+\frac12}}
\, ds
\Big)^2\,dt \\
& \le c\int^T_0 \Big[ \frac1{t^{2\beta-1+2\mu}} + t^{2\beta-1} \times
\\ & \qquad
\Big(\int^t_0
\frac{\n B^{-1}S(\frac{s}{4})\n \n(z_0-A)^{-\d}(S(t-\frac34s)-S(\frac{s}{4}))\n
_H\n (z_0-A)^\d S(\frac{s}{2})
x\n }{(t-s)^{\beta+\frac12}s^{\beta-\frac12}}\,ds \Big)^2 \Big] \, dt\\
& \le c\int^T_0 \Big[ \frac1{t^{2\beta-1+2\mu}}+t^{2\beta-1}
\Big( \int^t_0 \frac{1}{(t-s)^{\beta+\frac12-\delta}
s^{\delta+\mu+\beta-\frac12}}\,dr\Big)^2\Big] \, dt
\end{align*}
for generic constants $c$, where we
used \eqref{(ii)} and analyticity of the semigroup $S$. It follows
that
\begin{align*}
\int^T_0\n \mathscr K_\b^{-1}\varphi(t)\n ^2_H \, dt &\le  c\int^T_0 \Big[
\frac1{t^{2\beta-1+2\mu}} + t^{2\beta-1}
\Big(\frac1{t^{2\beta+\mu-1}}\Big)^2 \Big] \, dt\\
&= c \int^T_0 \frac1{t^{2\beta -1 + 2\mu}} \, dt,
\end{align*}
which is finite since $\mu < 1-\beta$. Therefore, the
control $\hat u(t) := - \frac1{T} B^{-1} S(t) x$ steering the solution from $x$
to zero belongs to $\mathscr H^\star$ as required.
\end{proof}

\begin{example}
[$2m-th$ order parabolic equation]\label{Y2}
We consider the problem
\begin{equation*}
\left\{
\begin{aligned}
\frac{\partial y}{\partial t}(t,\xi) &=(L_{2m}y)(t,\xi) + \eta^{\beta}(t,\xi), \
&& (t,\xi)\in (0,T)\times D, \\
y(0,\xi) &= x(\xi), && \xi\in D,
\end{aligned}
\right.
\end{equation*}
with the Dirichlet boundary conditions
\[
\frac{\partial^ky}{\partial \nu^k} (t,\xi) = 0, \quad k = 0,\dots,m-1, \
(t,\xi)\in
(0,T) \times \partial D
\]
where $\frac{\partial}{\partial \nu}$ denotes the co-normal derivative, $D\subset
\mathbb R^d$ is a bounded domain with a smooth boundary and
\[
L_{2m} = \sum_{|\alpha|\le 2m} a_{\alpha}(\xi) D^{\alpha}
\]
with $a_{\alpha}\in C^{\infty}_{\rm b} (D)$ is a uniformly elliptic operator on
$D$.
We let $A$ denote its realisation in  $E=L^2(D)$.
The Gaussian noise is fractional in time
and is modelled as
\[
\eta^{\beta} (t,\xi) = \Big(B\frac{\partial}{\partial t} W^{\beta}
(t,\cdot)\Big)(\xi),
\]
where $B$ is a given bounded injective operator on $H = E = L^2(D)$.

Suppose that condition (ii) of the theorem holds with exponent $\mu\in [0,
1-\b)$ (where we may put $z_0 = 0$), suppose that there exists $\mu ' \in [0,\mu
] $ such that
\begin{equation}\label{range2}
{\rm Range}(B) \subset {\rm Dom}(-A)^{\mu'}
\end{equation}

 and assume that
\begin{equation}\label{dm}
\frac{d}{4m} < \beta + \mu'.
\end{equation}
Then condition (i) of the theorem is satisfied as well. Indeed, as is well
known,
$(-A)^{-\rho}$ is Hilbert-Schmidt for any
$\rho>\frac{d}{4m}$, so
\begin{align*}
\n S(t)B\n_{\mathscr L_2(H,E)}
& \le c \n S(t)(-A)^{-\mu'}\n _{\mathscr L_2(E)} \\ & \le
c\n (-A)^{-\rho} \n _{\mathscr L_2(E)} \ \cdot \n (-A)^{\rho-\mu'}S(t)\n
_{\mathscr
L(E)} \le c \frac{c}{t^{\rho-\mu'}}.
\end{align*}
It follows that (i) is satisfied. Thus,
equivalence of the laws of $Z_T^x$, $x\in E$, is obtained if \eqref{dm} and the
condition (ii) of the theorem hold.
Note that \eqref{dm} is always satisfied if $\frac{d}{4m} < \beta$.

For example, for the stochastic heat equation $(m=1)$ with the noise fractional
in time and white in space $(\mu=\mu' =0$) the result holds if $\beta
>\frac{d}{4}$.
\end{example}

The condition (ii) of Theorem \ref{thm:Y1} is just sufficient and not
necessary, as can be seen in the following example.

\begin{example}\label{Y3}
In this example we take $H = E$ to be a separable real Hilbert space with
orthonormal basis $(e_n)_{n\ge 1}$ and define the operators $A$ and $B$ by
\[
Ae_n = - \alpha_n e_n, \ \ Be_n = \sqrt{\lambda_n} e_n, \ \ \quad n\ge 1,
\]
with $0 < \lambda_n \le \lambda_0$ and $0 < \alpha_1 \le \a_2\le \dots
\rightarrow \infty$.
Let $H_T$ denote the reproducing kernel Hilbert space associated with the
covariance operator $Q_T$ of the Gaussian
random variable $Z_T =
\int^T_0 S(T-t) B\,dW^{\beta}_t$. By the results of Section \ref{sec:2},
equivalence of laws
for $Z_T^x$, $x\in E$, holds if and only if the range of $S(T)$ is
contained in $H_T$. Since $S(T)$ is self-adjoint,
this happens if and only if there is a constant $c\ge 0$ such that
\[
\n S(T) x\n ^2 \le c \langle Q_Tx,x\rangle,  \quad x\in E.
\]
Under the above assumptions, this is equivalent to
\begin{equation}\label{N}
e^{-2\alpha_n T} \le c q_n, \ \ n\ge 1,
\end{equation}
where $Q_T e_n = q_n e_n$. It is easily seen that \eqref{N} holds if and only if
\begin{equation}\label{necsuf}
\frac{\alpha^{2\beta}_n}{\lambda_n} e^{-2\alpha_n T}
\end{equation}
is bounded. This follows from the fact that there exist some constants $C_1,C_2
\ge 0$ such that
$$
c_1 \frac{\lambda_n}{\alpha^{2\beta}_n} \le q_n \le c_2
\frac{\lambda_n}{\alpha^{2\beta}_n}
$$
For $\beta > \frac12$, this was proved in \cite{DMP1}. If $0 < \beta < \frac12$
we have
\[
q_n (T) = \langle Q_Te_n,e_n\rangle = \lambda_n \int^T_0 \n \mathscr K_\b^*
\psi_n(t)\n_H ^2 \, dt
\]
where $\psi_n (t) = e^{-\alpha_nt}$.
Furthermore,
\begin{align*}
\int^T_0 \n \mathscr K_\b^*\psi_n (t)\n_H ^2 \,dt \cong\n \psi_n\n
^2_{H^{\frac12-\beta}(0,T)} & =
\int^T_0 \int^T_0 \frac{|e^{-\alpha_nr}-e^{-\alpha_ns}|^2}{|r-s|^{2-2\beta}}
\,dr\,ds\\
&= \frac1{\alpha_n^{2\beta}} \int^{\alpha_nT}_0 \int^{\alpha_nT}_0 \xi(t,s)\,dt
\,ds
\end{align*}
where $\xi(t,s) = \frac{|e^{-t} - e^{-s}|^2}{|t-s|^{2-2\beta}}$, and the
conclusion easily follows.

Clearly, \eqref{necsuf} is implied by, but does not necessarily imply, condition
(ii) of the theorem.
\end{example}

\begin{example}\label{ex:S0}
Consider the 1D stochastic parabolic equation with inhomogeneous
Neumann boundary conditions of fractional noise type, formally written as
\begin{equation}\label{eq:S1}
\left\{
\begin{aligned}
\frac{\partial y}{\partial t} (t,\xi) & = \frac{\partial}{\partial
\xi}\Big(p(\xi)\frac{\partial y}{\partial \xi}\Big)(t,\xi) + q(\xi)y(t,\xi), &&
(t,\xi)
\in (0,T) \times (0,1),\\
y(0,\xi) & = x(\xi), &&  \xi \in (0,1), \\
\frac{\partial y}{\partial \xi} (t,0) & = \sigma_0 \dot B_1^{\beta} (t), \
\frac{\partial y}{\partial \xi} (t,1) = \sigma_1 \dot B_2^{\beta} (t), && t\in
(0,T),
\end{aligned}
\right.
\end{equation}
where $W^{\beta}=(B^{\beta}_1,B^{\beta}_2)$ is a two-dimensional standard fBm
with the Hurst parameter $\beta \in (0,1)$, $\sigma_0,\sigma_1$ are real
constants, the initial datum $x$ belongs to $L^2(0,1)$, $p\in C^2[0,1]$ is
strictly positive,
and $q \in C[0,1]$.

It is standard to rewrite the formal equation \eqref{eq:S1} in the
infinite-dimensional form considered in the previous section (see e.g.
\cite{DaZ2, DMP1} and references therein),
\begin{equation}\label{eq:S2}
\left\{
\begin{aligned}
dX_t &= AX_t\, dt + B\,dW_t^{\beta}, \quad t\in [0,T], \\
X_0 &= x.
\end{aligned}
\right.
\end{equation}
Here $A$ is the realisation in
in the space $E = L^2(0,1)$ of the partial differential operator
$$
A = \frac{\partial}{\partial \xi} \left(p(\cdot)\frac{\partial}{\partial \xi}
\right) + qI
$$
with domain
$$
\text{Dom}(A) = \{\varphi \in E; \varphi,\varphi' \ \text{absolutely
continuous}, \ \varphi''\in E, \ \varphi'(0) = \varphi'(1) = 0\}.$$
The operator $A$ is uniformly elliptic with homogeneous Neumann boundary
conditions and it is well known that it generates a strongly continuous and
analytic semigroup $S$ on $E$.
In order to define the operator $B$ we set $H = \R^2$.  Fix a constant $k> \max
q$ and consider the
second order boundary value problem
$$
\begin{aligned}
kz - Az & = 0 \  \ \text{on} \ (0,1)\\
\frac{\partial z}{\partial \xi} (j) & = u_j, \ \ j\in \{0,1\},
\end{aligned}
$$
for $(u_0,u_1) \in\mathbb R^2$. This problem has a unique solution for
every $(u_0,u_1) \in\mathbb R^2$. The Neumann map $N: (u_0,u_1) \mapsto z$ is an
element of the space $\mathscr L(\mathbb R^2, E^{\varepsilon})$ for arbitrary
$\varepsilon < \frac34$ (see \cite{DaZ2, MPi}
for details). Setting $B= \hat AN$, $\hat A \in\mathscr
L(E^{\varepsilon}, E^{\varepsilon-1})$ is the isomorphic extension of the
operator $kI-A$, we thus have $B\in \mathscr L(\mathbb R^2, E^{\varepsilon
-1})=\mathscr L(H,E^{\varepsilon -1})$. Now, the infinite-dimensional form (the
mild solution) of the equation \eqref{eq:S2} reads
\begin{equation}\label{eq:S3}
X_t^x = S(t) x + \int^t_0 S(t-r) B\,dW^{\beta}(r), \quad t\in [0,T],
\end{equation}
(here the extension of the semigroup $S(t),\ t>0$ to the space $\mathscr L
(E^{\epsilon -1},E)$ is denoted again by $S(t)$) which fits in the framework
developed in Section \ref{sec:2} with the spaces $H = \R^2$
and $E_1 = E^{\varepsilon -1}$ (for more detailed justification of this
approach we refer to \cite{DaZ2, DMP1, MPi}).

Assume that $\beta \in (\frac14,1)$ and fix $\varepsilon \in (1-\beta,
\frac34)$. Then
by the analyticity of the semigroup $S$,
$$
\n S(t)B\n_{\mathscr L_2(H,E)} \le c_1 \n S(t)\n _{\mathscr L(E^{\varepsilon
-1}, E)}
\n B\n _{\mathscr L(\mathbb R^2,E^{\varepsilon -1})} \le c_2 t^{\varepsilon -1},
\quad
t\in (0,T],
$$
for some constants $c_1,c_2$, so by \cite{DMP1} (or Proposition \ref{XE})
the mild solution
\eqref{eq:S3} is a well-defined $E$-continuous process. Our aim is to
investigate the equivalence of probability distributions $\mu^x_T$ of
solutions $X_T^x$ corresponding to initial datum $x\in E$. To this end, we use
the null
controllability result of Fattorini and Russel \cite{FR} for the controlled
equation
\begin{equation}\label{eq:S3a}
\left\{\begin{aligned}
\frac{\partial y}{\partial t} (t,\xi) & = (Ay)(t,\xi),  && t\in (0,T), \ \xi\in
(0,1),\\
y(0,\xi) & = x(\xi), && \xi\in (0,1), \\
\frac{\partial}{\partial\xi}y(t,0) & = u_0(t), \
\frac{\partial}{\partial\xi}y(t,1) = u_1 (t), && t\in (0,T),
\end{aligned}
\right.
\end{equation}
with $u = (u_1,u_2)\in L^2(0,T)$.
 First, note that the operator $A$ is self-adjoint and possesses a sequence
$(-\lambda_n)_{n\ge 1}$ of real eigenvalues such that
$$
\lambda_1 < \lambda_2 < \lambda_3  < \dots
$$
and $\lim_{n\rightarrow \infty}\lambda_n = +\infty$. Moreover, there is a real
constant $\alpha$ such that
\begin{equation}\label{eq:S4}
\lambda_n = \frac{\pi^2}{L^2} (n+\alpha)^2 + O(1), \ n \rightarrow \infty,
\end{equation}
where $L= \int^1_0 p^{-\frac12}(z)\,dz$ (cf. \cite{A, R}). Denoting by
$(e_n)_{n\ge 1}$ the normalized eigenfunctions corresponding to
$(\lambda_n)_{n\ge 1}$, the
solution to the equation \eqref{eq:S3a} may be expressed by the expansion $y(T)
= \sum^{\infty}_{n=1} y_n (T) e_n$ in $L^2(0,1)$, where
\begin{equation}\label{eq:S5}
y_n(T) = e^{-\lambda_nT} x_n + \int^T_0 e^{-\lambda_n(T-t)} \beta^n_0 u_0
(t)\,dt
+\int^T_0 e^{-\lambda_n(T-t)} \beta^n_1 u_1 (t)\, dt
\end{equation}
with $x_n = \langle x,e_n\rangle$ and
$$
\beta^n_0 = -p(0)\sigma_0 e_n(0), \ \beta^n_1 = -p(1)\sigma_1 e_n(1)
$$
(cf. \cite{FR}). It is clear that \eqref{eq:S3a} is not controllable in any
usual sense if $\sigma_0 = \sigma_1 =0$. Assume that at least one of the
constants $\sigma_0,\sigma_1$ is not zero (for instance, $\sigma_0 \not=0$). Set
$$
c_n = \frac{e^{-\lambda_nT}x_n}{\beta_0^n}.
$$
It is known that $\beta^n_0 \sim \text{const} \sqrt{|\lambda_n|}$ (cf.
\cite{FR}). Therefore, taking $\eta > 0$ arbitrary:
$$
\sum^{\infty}_{n=0} |c_n \lambda_n| \text{exp} \{(L+\eta)\sqrt{|\lambda_n|}\}
\le c \sum^{\infty}_{n=0} |x_n| \sqrt{|\lambda_n|} e^{-\lambda_nT}
\text{exp} \{(L+\eta)\sqrt{|\lambda_n|}\} < \infty
$$
by \eqref{eq:S4}. Hence by \cite[Corollary 3.2]{FR},
 there exists a solution $h\in L^2(0,T)$ to the moment problem
$$
\int^T_0 e^{-\lambda_nt}h(t) \,dt = \lambda_n c_n, \quad n\ge 1.
$$
It follows that $u_0(t):=\int^t_0 h(s)\, ds$ solves the moment problem
$$
\int^T_0 e^{-\lambda_n t} u_0 (t)\,dt = c_n
$$
and therefore, the control $u(t) = (u_0(t),0)$ steers the point $x$ to zero at
$t=T$. Obviously, $u\in W^{1,2}(0,T,\mathbb R^2)$ and it is easy to verify that
$W^{1,2}(0,T,\mathbb R^2)\subset\mathscr H^\star$ for each value of the Hurst
parameter $\beta \in (0,1)$. Summarizing, by Theorems \ref{thm:NC1} and
\ref{thm:NC2} we obtain that for each $\beta \in (0,1)$ the measures $\mu^x_T$,
$x\in E$,
are equivalent whenever $\sigma^2_0 +\sigma^2_1 \not=0$.
\end{example}

\section{Strong Feller property for semilinear equations}\label{sec:4}

In this section we present some applications of the general results from the
previous part
to stochastic semilinear equations with additive fractional noise. It
is shown that the null controllability of the deterministic equation
\begin{equation}\label{R00}
 y'=Ay +Bu
\end{equation}
 in the appropriate sense is equivalent to the continuous dependence of
probability laws of solutions to the corresponding stochastic semilinear
equation on the initial datum. The latter property in the Markovian case is
called the {\em strong Feller property}.

Consider the semilinear equation
\begin{equation}\label{eq:R1}
\left\{
\begin{aligned}
dX_t^x &= (AX_t^x + F(X_t^x)\,dt + B\,dW^{\beta}_t, \quad t\in (0,T),\\
X_0^x &= x
\end{aligned}
\right.
\end{equation}
in the space  $E$, which is here for simplicity assumed to be a real separable
Hilbert space (cf. Remark \ref{rem:R29} below). The operators $A$ and $B$ and
the
driving noise $W^{\beta}_t$ are
the same as in the linear case (i.e. $W^{\beta}_t$ is an $H$-isonormal Gaussian
process representing the $H$-cylindrical fBm with the Hurst parameter
$\beta \in (0,1)$). The operator $A: \text{Dom}(A)\subset E\rightarrow E$ is
assumed to generate an analytic semigroup $S=(S(t))_{t\ge 0}$ on $E$ and the
condition
\eqref{eq:X8} of Proposition \ref{XE} is supposed to hold. Under these
assumptions,
for each initial datum $x\in E$ the mild solution
$(Z^x_t)_{t\in [0,T]}$ of the linear equation \eqref{eq:X77} exists and has a
modification
with paths in $C^{\delta}([0,T],E)$ for all $0 \le \delta < \beta -\lambda$.
Let us now in addition assume that $B\in\mathscr L(H,E)$ is injective and that
the range of the
nonlinear function $F: E
\rightarrow E$ is contained in the range of $B$. This allows us to define the
function $G: E \rightarrow H$,
$G := B^{-1} F$. We impose the following conditions on $G$:

\begin{itemize}
\item[(G)] The function $G: E \rightarrow H$
 is continuous and has at most linear
growth, i.e.
\begin{equation}\label{eq:R2}
\n G(x)\n _H \le k (1+\n x\n _E), \quad x\in E,
\end{equation}
for some $k\ge 0$. If $\beta > \frac12$ we make the additional H\"older
continuity assumption
\begin{equation}\label{eq:R3}
\n G(x) - G(y)\n _H \le k\n x-y\n ^{\alpha}_E, \quad x,y \in E,
\end{equation}
for some $k\ge 0$ and
\begin{equation}\label{eq:R4}
\alpha > \frac{\beta-\frac12}{\beta-\lambda}.
\end{equation}
\end{itemize}

Define the integral operator $\mathbb K_{\beta}$ induced by the kernel
$K_\b(t,s)$
(cf. \eqref{eq:2.1}),
\begin{equation}\label{eq:R5}
(\mathbb K_{\beta}\varphi)(t) =\int^t_0 K_\b(t,s) \varphi(s)\,ds
\end{equation}
for $\varphi \in L^2(0,T,H)$. By \cite{c31}, the operator
$$
K_{\beta} : L^2(0,T,H)\rightarrow I_{0+}^{\beta+\frac12}(L^2(0,T,H))
$$
is a bijection and its inverse $\mathbb K_\b^{-1}$ may be expressed,
for $\varphi \in
I^{\beta+\frac12}_{0+}(L^2(0,T,H))$, as
\begin{equation}\label{eq:R6}
(\mathbb K_\b^{-1} \varphi)(t) = c_{\beta} t_{\frac12-\beta}
D_{0+}^{\frac12-\beta} (t_{\beta-\frac12} D^{2\beta}_{0+}\varphi)(t)
\end{equation}
for $\beta \in (0,\frac12)$ and
\begin{equation}\label{eq:R7}
(\mathbb K_\b^{-1}\varphi)(t) = c_{\beta} t_{\beta-\frac12}
D^{\beta-\frac12}_{0+} (t_{\frac12-\beta}D\varphi)(t)
\end{equation}
for $\beta \in (\frac12,1)$. If moreover $\varphi \in W^{1,2} (0,T,H)$
we have
\begin{equation}\label{eq:R8}
(\mathbb K_\b^{-1} \varphi)(t) = c_{\beta} t_{\beta-\frac12}
I^{\frac12-\beta}_{0+} (t_{\frac12-\beta}D\varphi) (t)
\end{equation}
for $\beta \in (0,\frac12)$ (here $c_{\beta}$ is a positive constant depending
only on $\beta \in (0,1)$) (cf. \cite{NuO}).

By \eqref{norm}, the Gaussian process $\widetilde W$ defined as $\widetilde
W(h) : = W^{\beta} ((\mathscr K_\b^*)^{-1}h)$, where $\mathscr K_\b^*$ is the
operator
defined in $\eqref{eq:2.6}$,  is isonormal on $\mathscr H = L^2(0,T,H)$, i.e. it
is the classical white noise (see also \cite{BHOZ}). The following result has
been proved in \cite{DMP3}.

\begin{proposition}\label{prop:R9}
Assume that \eqref{eq:X8} and $(G)$ are satisfied. Then the
equation $\eqref{eq:R1}$ has a weak (in the probabilistic sense) solution
$(X(t))$ satisfying $X(0)=x$ which is weakly unique. Moreover, for each $x\in E$
and $T>0$ the probability laws $\mu^x_T$ and $\nu^x_T$ are equivalent, where
$\mu_T^x= \text{\rm Law} (Z^x_T)$ and $\nu^x_T = \text{\rm Law}(X_T^x)$,
and the density is given by
\begin{equation}\label{eq:R10}
\widetilde {\mathbb E} \varphi (X_T^x) = \mathbb E \varphi (Z^x_T) \rho_T(x)
\end{equation}
where $\widetilde {\mathbb E}$ is the expectation with respect to the
probability space
where the process $(X_t^x)$ is defined, $\varphi : E \rightarrow \mathbb R$ is
bounded Borel measurable and
\begin{equation}\label{eq:R11}
\begin{aligned}
\rho_T (x) := \text{\rm exp}\Big\{ &\int^T_0 \Big\langle \mathbb
K_\b^{-1} \Big( \int^{\cdot}_0 G(Z^x_s)\,ds\Big)(t)\, d\widetilde
W\Big\rangle_H \\
&-\frac12 \int^T_0 \Big\n \mathbb K_\b^{-1}
\Big(\int^{\cdot}_0G(Z^x_s)\,ds\Big)(t)\Big\n ^2_H\,dt\Big\}. 
\end{aligned}\end{equation}
\end{proposition}

Let $\mathscr B(E)$ denote the $\sigma$-algebra of Borel sets in $E$ and $\tau$
the topology of pointwise convergence in the space of finite signed measures on
$\mathscr B(E)$. Thus a net $(\mu_{\gamma})_{\gamma\in \Gamma}$ converges to
$\mu$ in $\tau$ if and only if $\lim_{\gamma\in\Gamma} \mu_{\gamma} (C)=\mu(C)$
for each $C\in\mathscr B(E)$. Now we formulate the main result of this section.

\begin{theorem}\label{thm:R12}
Assume $\eqref{eq:X8}$ and $(G)$. Then for each $T>0$ the following statements
are equivalent:

\begin{itemize}
\item[(i)] The controlled deterministic system $\eqref{R00}$ is null
controllable in time $T$.
\item[(ii)] The measures $\mu^x_T$, $x\in E$, are equivalent.
\item[(iii)] The measures $\nu^x_T$, $x\in E$, are equivalent.
\item[(iv)] $\nu^{x_n}_T \rightarrow \nu^x_T$ in $\tau$ whenever $x_n\rightarrow
x$ in $E$ (strong Feller property).
\end{itemize}
\end{theorem}

In the proof of Theorem \ref{thm:R12} we use the following lemma:

\begin{lemma}\label{lem:R13}
Assume $\eqref{eq:X8}$ and $(G)$. Then for each $T>0$
\begin{equation}\label{eq:R14}
\rho_T (x_n) \rightarrow \rho_T(x) \ \text{in} \ \mathbb P \ \text{if} \
x_n\rightarrow x \ \text{in} \ E.
\end{equation}
\end{lemma}

\begin{proof}
From the proofs of  \cite[Theorems 3.3 and 3.4]{DMP3} it easily follows that
$$
\mathbb E \int^T_0 \Big\n  \mathbb K_\b^{-1} \Big(\int_0^{\cdot}
G(Z^x_s)\,ds\Big)(t) \Big\n ^2_H\,dt <\infty, \quad x\in E.
$$
Hence it suffices to show that
\begin{equation}\label{eq:R15}
\mathbb E\int^T_0 \Big\n  \mathbb K_\b^{-1}
\Big(\int_0^{\cdot}(G(Z^{x_n}_s)- G(Z^x_s))\,ds\Big)(t)\Big\n ^2_H\,dt
\rightarrow
0
\end{equation}
whenever we have $x_n \rightarrow x$ in $E$. We will show \eqref{eq:R15}
separately for the cases $\beta \in (0,\frac12)$ and $\beta \in (\frac12,1)$.

First, consider the case $\beta \in(0,\frac12)$. By \eqref{eq:R8} we have
\begin{equation}\label{eq:R16}
\begin{aligned}
&\mathbb E  \Big\n  \mathbb K_\b^{-1} \Big( \int^{\cdot}_0 (G(Z_s^{x_n}) -
G(Z^x_s))\,ds \Big)\Big\n ^2_{L^2(0,T,H)} \\
&\le c^2_{\beta} \mathbb E \int^T_0 \Big( s^{\beta-\frac12}\Big\n  \int^s_0
r^{\frac12-\beta} (s-r)^{-\frac12-\beta} (G(Z^{x_n}_r-G(Z^x_r))\,dr\Big\n _H
\Big)^2\,ds.
\end{aligned}
\end{equation}
By continuity of $G$ we have $G(Z^{x_n}_r)-G(Z^x_r)\rightarrow 0$ for each $r\in
[0,T]$ $\mathbb P$-almost surely, and \eqref{eq:R2} yields
\begin{equation}\label{eq:R17}
\n G(Z^{x_n}_r)-G(Z^x_r)\n _H \le 2k (1+ \n S(r)x_n\n _E + \n Z_r^0\n ) \le
L(1+\n \widetilde Z\n _{C([0,T],E)})
\end{equation}
where $L\ge 0$ is a constant independent of $n\in\mathbb N$ and $r\in [0,T]$.
Hence we obtain \eqref{eq:R15} by the dominated
convergence theorem.

Now, consider the case $\beta \in (\frac12,1)$. By \eqref{eq:R7} it follows that
\begin{equation}\label{eq:R18}
\begin{aligned}
&\mathbb E \Big\n  \mathbb K_\b^{-1}\Big( \int^{\cdot}_0(G(Z^{x_n}_s) -
G(Z^x_s))\,ds\Big)\Big\n ^2_{L^2(0,T,H)} \\
&\le c^2_{\beta} \mathbb E\int^T_0 \Big\n
\frac{s^{\beta-\frac12}}{\Gamma(\frac32-\beta)} \Big(
\frac{s^{\frac12-\beta}(G(Z_s^{x_n})-G(Z^x_s))}{s^{\beta-\frac12}} \Big)
\\
&+(\beta-\tfrac12) \int^s_0 \frac{s^{\frac12-\beta}(G(Z_s^{x_n})-G(Z^x_s)) -
r^{\frac12-\beta}(G(Z^{x_n}_r)-G(Z^x_r))}{(s-r)^{\beta+\frac12}}dr \Big\n
^2_H\,ds.
\end{aligned}
\end{equation}
As in the previous case, by
the continuity of $G$ we have $G(Z_s^{x_n})-G(Z^x_s)\rightarrow 0$ for each
$s\in [0,T]$ $\mathbb P$-almost surely. We aim at showing \eqref{eq:R15} by the
dominated convergence theorem. By \eqref{eq:R17} we immediately obtain
\begin{equation}\label{eq:R19}
\mathbb E\int^T_0 \Big\n  \frac{s^{\beta-\frac12}}{\Gamma(\frac32-\beta)}
\frac{s^{\frac12-\beta}(G(Z_s^{x_n})-G(Z^x_s))}{s^{\beta-\frac12}}\Big\n
^2_H\,ds
\rightarrow 0, \ n\rightarrow \infty.
\end{equation}
Furthermore, we have (for a generic constant $c$)
\begin{equation}\label{eq:R20}
\begin{aligned}
&\mathbb E \int^T_0 \Big\n  \int^s_0
\frac{s^{\frac12-\beta}(G(Z^{x_n}_s)-G(Z^x_s))-r^{\frac12-\beta}(G(Z^{x_n}
_r)-G(Z^x_r))}
{(s-r)^{\beta+\frac12}}dr \Big\n ^2_H \,ds \\
&\le c\mathbb E\int^T_0 \Big( s^{\beta-\frac12} \int^s_0
\frac{s^{\frac12-\beta}-r^{\frac12-\beta}}{(s-r)^{\beta+\frac12}} \n
G(Z_r^{x_n})
- G(Z^x_r)\n _H\,dr \Big)^2\,ds \\
&\qquad +c \mathbb E\int^T_0 \Big(\int^s_0 \frac{\n G(Z_s^{x_n})-G(Z^{x_n}_s) -
(G(Z_r^{x_n}) - G(Z^x_r))\n _H}{(s-r)^{\beta+\frac12}}dr\Big)^2\,ds.
\end{aligned}
\end{equation}
The first integral on the right-hand side of \eqref{eq:R20} clearly tends to
zero by the
H\"older continuity condition \eqref{eq:R3} and the inequality
$$
\int^s_0 \frac{s^{\frac12-\beta}-r^{\frac12-\beta}}{(s-r)^{\beta +\frac12}}\,dr
\le
cs^{1-2\beta}, \ s\in(0,T).
$$
Again by \eqref{eq:R3} and analyticity of the semigroup $S$ we have
\begin{align*}
&\frac{\n G(Z_s^{x_n})-G(Z^x_s)-G(Z_r^{x_n}) + G(Z^x_r)\n
_H}{(s-r)^{\beta+\frac12}}
\\
&\qquad \le c \frac{\n S(s)x_n-S(r)x_n\n ^{\alpha}_H + \n S(s)x-S(r)x\n
^{\alpha}_H+
\n \widetilde Z_s - \widetilde Z_r\n ^{\alpha}_H}{(s-r)^{\beta+\frac12}} \\
&\qquad \le c\frac{(s-r)^{\alpha\varepsilon}+(s-r)^{\alpha\delta}\n \tilde
Z\n _{C^{\delta}([0,T],H)}}{r^{\alpha\varepsilon}(s-r)^{\beta+\frac12}}
\end{align*}
where the generic constant $c\ge 0$ does not depend on $n\in\mathbb N$ and
$s,r$,
$s>r$, $s,r \in (0,T)$, and $\varepsilon$ is such that $\alpha\varepsilon < 1$,
$\beta + \frac12 -\alpha\varepsilon < 1$ and $\delta \in
(\frac1{\alpha}(\beta-\frac12),\beta -\lambda)$ (note that this choice is
possible by \eqref{eq:R4} and the fact that $\widetilde Z \in
C^{\delta}([0,T],E))$. This
gives us the convergent majorant for the second integral on the right-hand side
of
\eqref{eq:R20}, and \eqref{eq:R15} follows by dominated convergence.
\end{proof}

\begin{remark}\label{rem:R20a}
Note that a sequence $(\mu _n)$ of Borel probability measures on $E$ is
conditionally sequentially compact in $\tau$ if and only if it is
equicontinuous, i.e.
\begin{equation}\label{eq:R21}
\lim_{k\rightarrow \infty} \sup_{n} \mu_n (A_k) = 0 \ \mbox{for all} \ (A_k)
\subset\mathscr B(E), \ A_k \searrow \emptyset
\end{equation}
 and therefore $\mu_n \rightarrow \mu$ in $\tau$ provided \eqref{eq:R21} and
$\mu_n \rightarrow \mu$ in the $w^*$-topology (that is, weakly in probabilistic
sense), cf. \cite[Theorem 2.6 and Lemma 3.15]{Ga}.

 Now we can complete the proof of Theorem \ref{thm:R12}.
\end{remark}

\begin{proof}
The equivalence (i) $\Leftrightarrow$ (ii) has been proved in Theorems
\ref{thm:NC1} and \ref{thm:NC2}. By Proposition \ref{prop:R9}, for each $T>0$
and
$x\in E$ the measures $\nu^x_T$ and $\mu^x_T$ are equivalent, so we have
trivially (ii) $\Leftrightarrow$ (iii). We prove (iv) $\Rightarrow$ (ii) by
contradiction. If (ii) is false there exist $x_0,x_1 \in E$ and $T>0$ such that
$\mu^{x_0}_T \perp \mu^{x_1}_T$ (Gaussian measures must be singular unless they
are equivalent). By the Feldman-H\'ajek Theorem, we then have $\mu^{x_n}_T \perp
\mu^0_T$ where $x_n = \frac1{n} (x_0-x_1)\rightarrow 0$. By Proposition
\ref{prop:R9} $\nu^{x_n}_T \perp \nu^0_T$ which contradicts (iv). It remains to
show (ii) $\Rightarrow$ (iv). The proof is based on the Lemma~\ref{lem:R13}
above and follows the idea from \cite{MS} (where the proof is given for Markov
case).

First, note that (ii) implies that $\mu^{x_n}_T \rightarrow \mu^x_T$ in $\tau$.
This easily follows from the Cameron-Martin formula (for the density of
$\mu_T^{x_n}$ with respect to $\mu^x_T$, cf. \cite{DaZ} for a similar proof).
Also, by
\cite[Theorem II.21]{DeM} and Lemma~\ref{lem:R13} we immediately obtain that
\begin{equation}\label{eq:R22}
\rho_T (x_n) \rightarrow\rho_T(x) \ \text{in} \ L^1(\Omega)
\end{equation}
and the sequence of densities $\rho_T(x_n)$ is equiintegrable. By Remark
\ref{rem:R20a} it is sufficient to prove
\begin{equation}\label{eq:R23}
\sup_n \nu^{x_n}_T (A_k) \rightarrow 0, \ k\rightarrow \infty
\end{equation}
for arbitrary $(A_k) \subset\mathscr B(E)$, $A_k \searrow \emptyset$, and
\begin{equation}\label{eq:R24}
\int \varphi (y) \,d\nu^{x_n}_T (y) \rightarrow \int \varphi (y) \,d\nu^x_T(y),
\
n\rightarrow \infty
\end{equation}
for each $C_b(E)$. We have
\begin{align*}
&\sup_n \nu^{x_n}_T (A_k) = \sup_n \mathbb E {\bf 1}_{A_k}(Z^{x_n}_T)\rho_T
(x_n) \\
&\qquad \le K \sup_n \mathbb E {\bf 1}_{A_n} (Z^{x_n}_T) + \sup_n \mathbb E {\bf
1}_{[\rho_T(x_n)>K]} \rho_T (x_n)
\end{align*}
for arbitrary $K>0$. Now, we have
$$
\sup_n \mathbb E {\bf 1}_{[\rho_T(x_n)>K]} \rho_T (x_n) \rightarrow 0, \ K
\rightarrow \infty
$$
by equiintegrability of $\rho_T(x_n)$ and
$$
\sup_n \mathbb E {\bf 1}_{A_k} (Z^{x_n}_T) \rightarrow 0, \ k \rightarrow \infty
$$
by Remark \ref{rem:R20a} applied to the linear equation, where we use the fact
that $\mu^{x_n}_T\rightarrow \mu_T^x$ in $\tau$, and \eqref{eq:R23} follows.
Furthermore,
\begin{align*}
& \Big| \int_E \varphi(y) \,d\nu^{x_n}_T(y) - \int_E \varphi(y)
\,d\nu^x_T(y)\Big|\\
& \qquad = |\mathbb E
\varphi (Z_T^{x_n})\rho_T(x_n) -\mathbb E\varphi(Z^x_T)\rho_T(x)|\\
&\qquad \le \mathbb E |(\rho_T (x_n)-\rho(x))\varphi (Z^{x_n}_T)| +\mathbb E
\rho_T (x)
|\varphi(Z^{x_n}_T)-\varphi(Z^x_T)| \\
&\qquad \le \sup \n \varphi\n \mathbb E |\rho_T(x_n)-\rho_T(x)| + K \mathbb E
|\varphi
(Z^{x_n}_T) - \varphi(Z^x_T)|
\\ &\qquad \qquad
 +2 \sup \n \varphi\n \mathbb E {\bf 1}_{[\rho_T(x)>K]} \rho_T (x),
\end{align*}
for arbitrary $K>0$. Clearly, we have $Z^{x_n}_T - Z^x_T = S(T)
(x_n-x)\rightarrow 0$ as $n\rightarrow \infty$, hence using \eqref{eq:R22} we
obtain \eqref{eq:R24}, which concludes the proof of the Theorem.
\end{proof}

\begin{example}\label{ex:R25}
Consider the 1D stochastic equation of reaction-diffusion type
\begin{equation}\label{eq:R26}
\left\{
\begin{aligned}
\frac{\partial y}{\partial t} (t,\xi) &= \frac{\partial^2y}{\partial \xi}
(t,\xi) + f(y,t,\xi)) +\eta^{\beta}(t,\xi), &&
(t,\xi)\in (0,T)\times(0,1)\\
y(0,\xi) &= x(\xi),  && \xi \in (0,1)\\
y(t,0) &= y(t,1) = 0, && t\in (0,T),
\end{aligned}
\right.
\end{equation}
where $f:\mathbb R\rightarrow\mathbb R$ and the noise $\eta^{\beta}$
is fractional in time with the Hurst parameter $\beta\in (0,1)$ and white in
space.
For the linear case $(f=0)$ it is a particular case of the Example \ref{Y2}. The
rigorous interpretation of \eqref{eq:R26} is the equation \eqref{eq:R1} where we
put $H=E= L^2(0,1)$, $B=I$, $A= \frac{\partial^2}{\partial \xi^2}$,
$$
\text{Dom}\,(A) = \{\varphi \in E; \ \varphi,\varphi' \ \text{are} \ AC,\
\varphi'' \in E, \ \varphi(0) = \varphi(1)= 0\}
$$
and $F:E \rightarrow E$, $(F(x))(\xi) := f(x(\xi))$, is the Nemytskii operator.
We need to impose some conditions on $f$ so that the operator $F:E \rightarrow
E$ is well-defined a satisfies the assumptions of the present section. To this
end we
assume  that $f$ is continuous and of at most linear growth,
\begin{equation}\label{eq:R27}
|f(\xi)| \le k(1+|\xi|), \quad \xi\in \mathbb R,
\end{equation}
for some $k\ge 0$, and if $\beta > \frac12$  we assume in addition that $f$ is
also H\"older
continuous,
\begin{equation}\label{eq:R28}
|f(\xi)-f(\eta)| \le k|\xi-\eta|^{\alpha}, \quad \xi, \eta\in \mathbb R,
\end{equation}
for some $k\ge 0$ and $\alpha>0$ satisfying $\alpha > \frac{\beta -
\frac12}{\beta-\frac14}$. Using the fact that $\lambda = \frac14$
(cf. Example
\ref{Y2}) this easily implies that the Hypothesis (G) is satisfied. As we have
seen in the Example \ref{Y2}, the correspond deterministic system \eqref{R00} is
null controllable in this case and we may conclude that for each $T>0$
the probability laws $\nu^x_T$, $x\in E$, are equivalent and the mapping
$x\mapsto \nu^x_T$ is continuous in the topology of pointwise convergence (the
strong Feller property holds).
\end{example}

\begin{remark}\label{rem:R29}
(i) In the previous example, the conditions \eqref{eq:R27}, \eqref{eq:R28}
(which
are clearly too strong for usual reaction-diffusion models) may be replaced by
\begin{equation*}
\begin{aligned}
|f(\xi)|& \le K_\b(1+|\xi|^{\rho}),\\
|f(\xi)-f(\eta)|\,\text{sign}(\xi-\eta) & \le K_\b(\xi-\eta), \quad \xi, \eta
\in \mathbb R
\end{aligned}
\end{equation*}
and (if $\beta >\frac12$)
\begin{equation}\label{eq:R32}
|f(\xi)-f(\eta)| \le K_\b(1+|\xi|^q+|\eta|^q)|\xi-\eta|^{\alpha}, \quad \xi,\eta
\in\mathbb R
\end{equation}
for some $K\ge 0$, $\rho, q >0$, and $\alpha \le 1$ satisfying $\alpha >
\frac{\beta-\frac12}{\beta-\frac14}$. This may be shown in the same way
as in the present example, but the proof of the analogue of Lemma \ref{lem:R13}
becomes technically
more complicated (and $E$ is no longer Hilbert space).

(ii) Using Example \ref{ex:S0} we may also consider the semilinear equation with
boundary noise
\begin{equation*}\left\{
\begin{aligned}
\frac{\partial y}{\partial t} (t,\xi)& =\frac{\partial}{\partial \xi}
\Big(p(\xi)\frac{\partial }{\partial \xi}y\Big)(t,\xi) + q(\xi) y(t,\xi),
&& (t,\xi)\in (0,T)\times (0,1),\\
y(0,\xi) &= x(\xi), &&  \xi\in (0,1),\\
 \frac{\partial}{\partial \xi}y(t,j) &= f_j(y(t,\cdot)) +
\sigma_j\dot B_j^{\beta}(t), && t\in (0,T), \ j\in\{0,1\},
\end{aligned}\right.
\end{equation*}
where $f_j : E\rightarrow \mathbb R$, $j\in\{0,1\}$, satisfy corresponding
continuity
and growth conditions. Unless both $\sigma_j$, $j\in\{0,1\}$, are zero we again
obtain that the strong Feller property holds for the above semilinear stochastic
equation.
\end{remark}

\end{document}